\documentclass[12pt]{article}
\usepackage[english]{babel}
\usepackage[latin1]{inputenc}
\usepackage{amsfonts,amssymb,amsmath, epsfig}
\usepackage{amsthm}
\usepackage{color,graphicx,graphics,psfrag}
\usepackage{amsmath,amstext,amssymb,amsfonts, amscd}

\usepackage{hyperref}          

\def\Box{\leavevmode\vbox{\hrule
     \hbox{\vrule\kern4pt\vbox{\kern4pt}%
           \vrule}\hrule}}

\newcounter{appendix}
\def\appendix{\advance\c@appendix by 1
   \def\thesection{\Alph{section}}
   \ifnum\c@appendix=1 \setcounter{section}{-1} \fi
   \@startsection {section}{1}{\z@}{-3.5ex plus -1ex minus 
   -.2ex}{2.3ex plus .2ex}{\Large\bf}}

\def\paragraph#1{{\bf #1\ }}

\newtheorem{lemma}{Lemma}[section]  

\newtheorem{theorem}[lemma]{Theorem}

\newtheorem{corollary}[lemma]{Corollary}

\newtheorem{proposition}[lemma]{Proposition}

\newtheorem{remark}{Remark}[section]

\newcommand{\eps}{\varepsilon}
\newcommand{\lp}{\left(}
\newcommand{\rp}{\right)}
\newcommand{\Hop}{\mbox{H}}
\newcommand{\R}{\mathbb{R}}

\title{Continuum dynamics of the intention field under weakly cohesive social interactions} 
\author{Pierre Degond$^{(1)}$, Jian-Guo Liu$^{(2)}$, Sara Merino-Aceituno$^{(3)}$, Thomas Tardiveau$^{(4)}$}  
\date{} 
\begin{document}

\maketitle


\begin{center}
1-Department of Mathematics,
Imperial College London, \\
London SW7 2AZ, United Kingdom\\
email: pdegond@imperial.ac.uk
\end{center}

\begin{center}
2-Department of Physics and Department of Mathematics\\
Duke University,
Durham, NC 27708, USA\\
email: jliu@phy.duke.edu
\end{center}

\begin{center}
3-Department of Mathematics,
Imperial College London, \\
London SW7 2AZ, United Kingdom\\
email: s.merino-aceituno@imperial.ac.uk
\end{center}

\begin{center}
4-\'Ecole Polytechnique\\
Paris, 91128 Palaiseau, Route de Saclay, France, \\
email:  thomas.tardiveau@polytechnique.edu
\end{center}

\vspace{0.5 cm}
\begin{abstract}
We investigate the long-time dynamics of an opinion formation model inspired by a work by Borghesi, Bouchaud and Jensen. Firstly, we derive a Fokker-Planck type equation under the assumption that interactions between individuals produce little consensus of opinion (grazing collision approximation). Secondly, we study conditions under which  the Fokker-Planck equation has non-trivial equilibria  and derive the macroscopic limit (corresponding to the long-time dynamics and spatially localized interactions) for the evolution of the mean opinion. Finally, we compare  two different types of interaction rates: the original one given in the work of Borghesi, Bouchaud and Jensen (symmetric binary interactions) and one inspired from works by Motsch and Tadmor (non-symmetric binary interactions). We show that the first case leads to a conservative model for the density of the mean opinion whereas the second case leads to a non-conservative equation. We also show that
the speed at which consensus is reached asymptotically for these two rates has fairly different density dependence.
\end{abstract}

\medskip
\noindent
{\bf Acknowledgements:}      
P.D. acknowledges support from the Royal Society and the Wolfson foundation through a Royal Society Wolfson Research Merit Award; the British ``Engineering and Physical Research Council''	 under grant ref: EP/M006883/1; the National Science Foundation under NSF Grant RNMS11-07444 (KI-Net).
P.D. is on leave from CNRS, Institut de Math\'ematiques de Toulouse, France. P.D. is grateful to J-P. Bouchaud for suggesting the problem and for stimulating discussions.

J.G.L. was partially supported by KI-Net NSF
RNMS grant No. 1107291 and NSF grant DMS 1514826.

 S.M.A. was supported by the British ``Engineering and Physical Research Council'' under grant ref: EP/M006883/1. 
 
 T.T. gratefully acknowledges the hospitality of the Department of Mathematics at Imperial College London, where this research was conducted.

\medskip
\noindent
{\bf Key words: } opinion formation; grazing limit; non-symmetric rate; continuum limit; Deffuant-Weisbuch model.

\medskip
\noindent
{\bf AMS Subject classification: }82C21, 82C22, 82C26, 82C31, 82C40, 82C70, 91B12, 91B14, 91B70, 91B72.
\vskip 0.4cm

\setcounter{equation}{0}
\section{Introduction}
\label{intro}

The goal of the present article is the investigation of an opinion formation model inspired from the one presented in Ref. \cite{Bouchaud_etal_JStatMech14}. Firstly, we obtain the mean-field equations for this model and approximate the dynamics under the assumption that interactions between individuals produce little convergence of opinions (\textit{weak consensus interaction}). We study the equilibria for this case and show that, under some conditions,  it corresponds to a Gaussian  distribution $\mathcal{N}(\varphi,\sigma^{2})$ with a fixed given variance $\sigma^2$ but undetermined mean $\varphi$. The final aim is to derive the equation for the evolution of the mean opinion $\varphi$ in the spatially heterogeneous case when interactions become localized. During this analysis, we will consider two different cases corresponding to two different types of interaction rates: the original one given in \cite{Bouchaud_etal_JStatMech14} and one inspired from Refs. \cite{Motsch_Tadmor_JSP11, Motsch_Tadmor_preprint}. We show that, asymptotically, the dynamics for the second rate reaches consensus faster in regions of low density of individuals while for high density regions, the dynamics corresponding to the first rate is faster in reaching consensus. The main result is discussed in the next section. As far as we know, this is the first result that derives the macroscopic dynamics for these equations. 

The tools used to carry out the present analysis are borrowed from kinetic theory and hydrodynamic limit techniques which originally were developed to tackle problems from Mathematical Physics. Recently these tools have found applications in the study of emergent phenomena in biological and social systems. Some illustrative examples of this are the study of self-organized collective behaviour in different settings like swarming and flocking \cite{Degond3,Lachowicz1},  fish schools \cite{Degond2, schooling}, ant trail formation \cite{Degond7} or collective cell migration \cite{cells};
evolution of traffic \cite{traffic1} and crowd dynamic \cite{crowd1,Bellomo}; the emergence of languages \cite{language}, cultures \cite{culture}, segregation \cite{schelling_model} or social classes \cite{ben4}. In particular, recently new approaches have been introduced to describe the formation of opinions using mean-field (or kinetic) equations \cite{Wolfram,Toscani1,politic2,meanfield}.

\bigskip

However, the modelling of opinion formation has a long history. One can trace it back to the Condorcet method for voting systems (1785) and more recently to the Fisher-KPP equation (1937) which has found applications in  modelling rumour spreading \cite{KPP}. This was followed in  1971 by Ref. \cite{oldest} that models polarization phenomena in  a society, and in 1982 with the emergence of the concept of sociophysic \cite{Galam2}. One of the first approaches to opinion formation consisted of describing society as a graph, or network, with individuals located at nodes and interacting with their neighbouring nodes.  Such a description fostered links with the Ising model where spins were replaced by an opinion variable. This first approach also encouraged to view opinions as discrete-valued variables.  The main difference between the proposed models  is the rule by which opinions evolve. The interested reader  can find examples in the
Voter's Model \cite{Castellano1}, the Majority Rule Model \cite{Majority_rule} and the Sznajd model \cite{Ising1}.
%
  In opposition to this, in the model treated here opinions are represented by continuum-valued variables and individuals are located in the continuum space $\R^n$.

In opinion formation two major opposing mechanisms are considered: on the one hand, interaction between individuals leading to some type of consensus, and on the other hand, noise that accounts for other factors like self-thinking, media,.... The balance between these two antagonist effects is key to the long-time evolution of opinions and the formation of large scale patterns like, for example, emergence of clusters.
For instance, \cite{Ben4} investigates the formation of clusters under the rule that interacting individuals  adopt the average value of their opinions. 
Here, we also study the balance between these two opposing effects, particularly, we give conditions for consensus to emerge and we study how given interaction rates affects the speed at which consensus is reached.

\bigskip
We will start first with  an extension of the Deffuant-Weisbuch model \cite{Deffiant_weisbuch}. Such models are  applied to the study of voters' intentions and their distance correlations \cite{Bouchaud4} \cite{Bouchaud3}.  For the asymptotic analysis, we consider the individual-based model presented in \cite{Bouchaud_etal_JStatMech14} (spatially heterogeneous version of \cite{Deffiant_weisbuch}) with an interaction rate inspired from
\cite{Motsch_Tadmor_JSP11, Motsch_Tadmor_preprint}. 
 In \cite{Bouchaud_etal_JStatMech14} the authors consider that the interaction rate between individuals is given by a centred Gaussian  evaluated at the distance between the opinions of the two interacting individuals. The authors showed that a phase transition emerges between social dissension and a socially cohesive phase with the mean opinion obeying a diffusion equation at the kinetic level. Here we focus our attention in the weak consensus approximation (i.e., the case where little consensus of opinions is reached after interactions). The equilibria and phase transitions that we obtain are consistent with the results in \cite{Bouchaud_etal_JStatMech14}.

There exist related works in the literature  that present different settings from ours. In Ref. \cite{Wolfram} a similar model, introduced by \cite{Toscani1}, is studied with a constant interaction rate and bounded domain.  In Ref. \cite{REF_DEGOND}, the authors investigate a kinetic equation close to ours on a periodic domain. In Ref. \cite{Toscani1}, the author considers a model where the outcome of a binary interaction depends on each of the individuals' intention but not on their difference of opinions. Finally, in Ref. \cite{Pareshi} the authors consider a model for  wealth dynamics where binary interactions are possible only if the outcome wealth remains positive.

%

\medskip
The paper is structured as follows. In the next section we discuss the main result, namely Theorem \ref{thm:diffusion_approx} where the evolution for the mean opinion dynamics is derived. In Sec. \ref{sec:rate} we present the individual-based model for opinion dynamics in the spatially homogeneous case and consider interactions leading to weak consensus (analogous to the so-called `grazing collision approximation' in gas dynamics). We study the equilibria of these dynamics. Finally, in Sec. \ref{sec:inhomogenous} we consider the spatially heterogeneous case for the previous model and derive the macroscopic equations for the evolution of the mean opinion.

\section{Discussion of the main results}
\label{sec:interpretation_results}

\subsection{Framework}
In the present paper we consider a model where pairs of individuals interact through their opinions or intentions  \eqref{eq:agent}. A parameter $\gamma$ measures how close both opinions become after an interaction. We assume here that $\gamma$ is small and therefore  little consensus is reached during interactions. We will consider that a noise with mean zero and variance $\Sigma^2$ is present in the interactions. Finally, these interactions take place at a given rate depending on how close the opinion of a pair of individuals is. Particularly, the rate is parametrized by $\zeta$ which represents the typical scale at which interactions take place. Under the conditions of Prop. \ref{prop:equi_gauss} we prove that the equilibrium for the corresponding mean-field equation is given by a Gaussian  with fixed variance $\sigma^2$ (given by $\gamma, \zeta, \Sigma^2$) and undetermined mean $\varphi$. In particular, we give criteria on $\gamma, \Sigma^2$ and $\zeta$ for a phase transition to occur between social consensus and social dissent, see Remark \ref{rem:phase_transition}.

We consider two types of rates given by Eqs. \eqref{eq:Hop} and \eqref{eq:rate_inhomo}. The first rate corresponds to symmetric interactions, i.e., the rate at which a pair $(i,j)$ interacts is the same as for the pair $(j,i)$. This rate depends only on the spatial distance and opinion distance between the pair of individuals (i.e., the closer individuals are in space and in opinion, the higher the rate at which they interact). In the other case, the interaction is non-symmetric. In this case, individuals forming a cluster in the space-opinion phase space interact very frequently whereas isolated individuals undergo fewer interactions and when they do interact they tend to interact with the close clusters. If $i$ denotes an isolated individual and $j$ an individual belonging to a close cluster then the influence of $j$ on $i$ will be larger than the influence of $i$ on $j$ by a ratio roughly equal the size of the cluster. This implies that the isolated individual $i$ changes its opinion towards a value closer to the opinion of individual $j$, while the opinion of $j$ does not change. This non symmetric relation is key to explain why this rate gives faster consensus  than the symmetric case in regions of low density, as we will explain later.

\subsection{Conservative properties and entropy}
The evolution for the mean opinion, or in the language of \cite{Bouchaud_etal_JStatMech14} of the ``intention field",  $\varphi=\varphi(\alpha,t)$ (where $\alpha$ is the spatial variable) is given for both rates in Th. \ref{thm:diffusion_approx} in the asymptotic time limit when interactions become localized. The density $\rho$ is constant in time and the density of opinion $\rho \varphi$ evolves according to
\begin{equation}
\label{eq:conservative_form_sym}
\frac{\partial}{\partial t}(\rho\varphi) = C_s \, \nabla_\alpha \cdot \lp \rho^2 \nabla_\alpha \varphi\rp, \quad C_s:=\gamma D \frac{\zeta^3}{(2\sigma_s^2 + \zeta^2)^{3/2}} 
\end{equation}
for the symmetric case and
\begin{equation}
\label{eq:cons_form_asym}
\frac{\partial}{\partial t} (\rho\varphi) = \frac{C_a }{\rho}\, \nabla_\alpha \cdot \lp \rho^2 \nabla_\alpha \varphi\rp, \quad C_a:=\gamma D \frac{\zeta^2}{\sigma_a^2+\zeta^2}.
\end{equation}
From these expressions we investigate conservative and homogenisation properties for the respective solutions.

To begin with, notice that,  for the symmetric case, the \textit{density of opinion} $\rho\varphi$ is a conserved quantity. This is expected since this is a conserved quantity at the kinetic level. The opposite holds true for the non-symmetric case: the non-conservation of $\rho\varphi$ at the kinetic level is maintained in the macroscopic dynamics. Nevertheless, we notice that in the non-symmetric case the value for $\rho^2\varphi$ is conserved (multiplying Eq. \eqref{eq:cons_form_asym} by $\rho$ as it is independent of time). Notice that if transport of individuals was taken into account on the kinetic equations this conservation property would most likely not hold (in particular, $\rho$ would become time-dependent).

Moreover, from Eqs. \eqref{eq:conservative_form_sym} and \eqref{eq:cons_form_asym} we can deduce the following entropy dissipation relations:
\begin{equation}
\label{eq:entropy_sym}
\frac{\partial}{\partial t}\int_{\R^n}\rho \varphi^2 \, d\alpha = - C_s \int_{\R^n} \rho^2 |\nabla_\alpha \varphi|^2\, d\alpha
\end{equation}
for the symmetric case and
\begin{equation}
\label{eq:entropy_asym}
\frac{\partial}{\partial t}\int_{\R^n}\rho^2\varphi^2\, d\alpha = - C_a \int_{\R^n} \rho^2 |\nabla_\alpha \varphi|^2\, d\alpha.
\end{equation}
This shows that for both cases the $L^2(\R^n)$ (weighted) norms on the left hand side decrease over time. When the time derivative reaches zero, the right hand side vanishes which implies that $\nabla_\alpha \varphi =0$ a.e., and so $\varphi$ is constant a.e.. Therefore, the dynamics tend to homogenise the value of the mean opinion, or intention field, $\varphi$.
Precisely absolute consensus takes place when $\varphi$ is constant, i.e., there is no difference in the mean opinion between different spatial regions. Observe that, indeed,  constant functions are stationary solutions to both equations. In the next section, we investigate for which one of the two rates considered consensus is reached faster asymptotically.

\begin{remark}[Conserved quantities and analogy with non-equilibrium thermodynamics]
Notice that the mean intention $\varphi$ is not a conserved quantity, only the density of intention $\rho\varphi$ is (in the symmetric case). Also, Eq. (\ref {eq:conservative_form_sym}) is consistent with Onsager's formalism of non-equilibrium thermodynamics where the time derivatives of the extensive variables (here the mean density of intention)  are balanced by the divergence of fluxes that are linear in the gradients of the intensive variables (here the mean intention itself). This is often referred to as the linear flux-force theory. There exists a duality between intensive and extensive variables through the entropy of the system \cite{Onsager}. This is why, in order to study conservative properties and entropy dissipation relations, we formulate Eqs. \eqref{eq:conservative_form_sym} and \eqref{eq:cons_form_asym}  in Onsager's formalism \cite{Onsager}, i.e.,  by considering the time derivative corresponding to the extensive variable $\rho \varphi$ instead of that of the intensive variable $\varphi$. However, in the non-symmetric case, the formalism of non-equilibrium thermodynamics does not strictly apply since the time derivative of the density of intention $\rho\varphi$ is not balanced by a divergence. Instead, there is the pre-factor $1/\rho$ in front of the divergence which makes the equation non-conservative. Thus, in this case, the standard entropy $\int \rho \varphi^2 \, dx$ as put forward by Onsager needs not be dissipated. But another entropy  $\int \rho^2 \varphi^2 \, dx$ is.
\end{remark}


To conclude this section, we remark that by Prop. \ref{prop:equi_gauss}  the diffusion constants can we rewritten as $C_s=(\sqrt{\zeta^2-\kappa}/\zeta)^3$ and $C_a=(2\zeta^2-\kappa)/(2\zeta^2)$ (with $\kappa= \Sigma^2/\gamma$) and so they stay positive for the same range values of $\kappa$  as that for which the Gaussian  equilibria given in Prop. \ref{prop:equi_gauss} is defined. This shows that the final model is well-posed.

\subsection{Comparison in the speed of consensus}
To compare the speed of consensus given by the two different rates, we recast again the Eqs. \eqref{eq:varphi_diffusion_sym}-\eqref{eq:varphi_diffusion_asym} from Th. \ref{thm:diffusion_approx} into, assuming that $\rho>0$,
\begin{equation} \label{eq:macro_intro}
\frac{\partial \varphi}{\partial t} - C_s\big( \rho \Delta_\alpha \varphi + 2\nabla_\alpha \rho \cdot \nabla_\alpha \varphi \big)=0, \quad C_s:=\gamma D \frac{\zeta^3}{(2\sigma_s^2 + \zeta^2)^{3/2}} 
\end{equation}
for the symmetric case and
\begin{equation} \label{eq:macro_asym_intro}
\frac{\partial \varphi}{\partial t} - C_a \left( \Delta_\alpha \varphi + \frac{2}{\rho}\,\nabla_\alpha  \rho \cdot \nabla_\alpha \varphi\right) =0, \quad C_a:=\gamma D \frac{\zeta^2}{\sigma_a^2+\zeta^2}
\end{equation}
for the non-symmetric case ($\sigma_s$ and $\sigma_a$ are given in Prop. \ref{prop:equi_gauss}). One can check that $C_a>C_s$. In both cases we have that the mean opinion, or intention field, diffuses and is transported in the direction $-\nabla_\alpha \rho$, i.e., from places of high concentration of individuals to places of low concentration. Since absolute consensus takes place when $\varphi$ is constant, the faster the mean opinion $\varphi$ is diffused and transported through space, the faster consensus is reached.

We start by examining regions with low population, say  $\rho\leq 1$. We observe that the diffusive coefficient for the symmetric case, given by $C_s\rho$, is always smaller than that of the non-symmetric case, corresponding to $C_a$ (since $C_a>C_s$). In particular, in this symmetric case, diffusion slows down in regions with lower density of individuals. Also the transport of the mean opinion towards areas of lower density is slower in the symmetric case since this speed corresponds to $C_s$ whereas in the non-symmetric case it corresponds to $C_a/\rho$. Notice then that, in the non-symmetric case, in low density regions this transport takes place very fast. However in the non-symmetric case, interaction rates are scale invariant with the local density while in the symmetric case they are homogeneous of degree one with the density. So, in the non-symmetric case, agents in low density regions interact at the same rate as in the high density regions while in the symmetric case, they interact at a much lower rate. The effect is opposite in the large density regions. 

In conclusion, from Th. \ref{thm:diffusion_approx} we deduce
\begin{corollary}
\label{cor:consensus}
Low density regions reach consensus faster in the non-symmetric rate case. By contrast, the large density regions reach consensus faster in the symmetric case. The crossover density between these two behaviours is given by $C_{s}\rho=C_{a}$, i.e. 
\begin{equation} \label{eq:crossover_density}
\rho=\frac{C_{a}}{C_{s}}.
\end{equation}
\end{corollary}

\medskip

This crossover in the speed to consensus can be understood heuristically as follows. In the symmetric case any pair of individuals have the same influence on each other. Contrary to this, in the non-symmetric case,
individuals belonging to a cluster (high density region) influence the opinion of isolated individuals (low density regions) but not the other way around, this is precisely due to the non-symmetric nature of the rate. As a consequence, in the non-symmetric case, because clusters have more influence on isolated individuals, isolated individuals adopt quickly the opinion of the cluster. This can be seen by examining the role of the convection term (the third term on the left hand side of \eqref{eq:macro_intro} and \eqref{eq:macro_asym_intro}). Indeed this convection term shows a transport of the mean opinion in the direction of the gradient for $\rho$,  i.e. from places of high concentration to places of low concentration. Actually this convection takes place much faster in low density regions in the non-symmetric case than in the symmetric case. Now, in high density regions, in the non-symmetric case, individuals of a cluster tend to interact very frequently among themselves but they interact less frequently with members of other clusters than in the symmetric case. This is why for higher density regions the symmetric rate reaches consensus faster: there are more interactions (exchange of opinions) between clusters.

\setcounter{equation}{0}
\section{Rate model and weak consensus approximation}
\label{sec:rate}

\subsection{Kinetic model: symmetric and non-symmetric cases}

We first consider a spatially homogeneous model for opinion formation. Denote by $\phi_i(t) \in {\mathbb R}$ the opinion of agent $i$ at time $t$; for example, it could represent the intention of agent $i$ to vote for a particular party as in \cite{Bouchaud_etal_JStatMech14}. When the pair of agents $(i,j)$ interact, they exchange opinions according to the following rule $(\phi_i,\phi_j) \to (\phi'_i,\phi_j)$ with 
\begin{equation}
 \phi'_i = \phi_i + \gamma (\phi_j - \phi_i) + \eta_i, 
\label{eq:agent}
\end{equation}
where $0 \leq \gamma \leq \frac{1}{2}$ is a constant which measures how opinions get closer after an interaction takes place. For each $i=1,\hdots,N$, where $N$ is the number of agents, $(\eta_i)_{i=1,\hdots,N}$ are independent, identically distributed random variables distributed according to a probability density $q=q(\eta)$ with zero mean, variance $\Sigma^2$ and decaying moments. This noise source accounts for other influences.
The interaction rate between a pair of individuals, given by $G_\zeta(\phi) =G(|\phi_i - \phi_j|/\zeta)$, is dependent on the distance between their opinions and is a decreasing function of its argument. In other words: individuals with far-away opinions interact at a lower rate (as in \cite{Deffiant_weisbuch}).
The parameter $\zeta$ is a typical scale in the sense that individuals with opinions distant by a quantity greater than $\zeta$ have a probability close to zero to interact. Here, we will consider two types of interaction rates which we write compactly as:
$$\frac{G_\zeta(|\phi_i-\phi_j|)}{\Hop\lp\frac{1}{N}\sum_{k=1}^N G_\zeta(|\phi_i-\phi_k|) \rp}$$
where $\Hop$ is a function that can correspond to either the constant 1 or the identity operator depending on the case we are interested in:
\begin{equation} \label{eq:Hop}
\begin{split}
& \ \ \ \ \ \ \ \ \ \Hop:\  g\in[0,\infty) \rightarrow \Hop(g) \in [0,\infty) \\ \\
&\Hop(g) = \left\{\begin{array}{ll}
1, \ \forall g \ge 0, \mbox{ i.e.} \ \Hop =1 & \mbox{(symmetric case)}\\
 \ \ \mbox{     or,}\\
\mbox{g}, \ \forall g \ge 0, \mbox{ i.e.} \ \Hop =\mbox{Id} & \mbox{(non-symmetric case)}.
\end{array} \right.
\end{split}
\end{equation}
 If $\Hop\equiv 1$, then the rate of interactions between two agents is symmetric. On the contrary, if $\Hop=\mbox{Id}$, the rate of interactions for  pair $(i, j)$ is not the same  as for pair $(j,i)$. 
In particular, the non-symmetric rate case is inspired by  the models in Refs. \cite{Motsch_Tadmor_JSP11, Motsch_Tadmor_preprint}.

We denote by $f=f(\phi,t)$ the probability distribution of agent's intentions $\phi$ at time $t$. According to the previous rule, and assuming propagation of chaos (see remark \ref{rem:SDE} below and Ref. \cite{Cergignani_illner}), the evolution of $f$ is given for any test function $g$, by: 
\begin{eqnarray}
&&\hspace{-1cm}
\frac{d}{d t} \int_{\R} f g \, d \phi = \int_{\R^3} \Big[ g \big(\phi + \gamma (\psi-\phi) + \eta \big) - g (\phi) \Big] \nonumber \\
&&\hspace{5cm}
 f(\phi,t) \, f(\psi,t) \, \frac{G_\zeta\lp |\phi-\psi|\rp}{\Hop(G_\zeta*f)}  \, q(\eta) \, d \eta \, d \phi \, d \psi.
\label{eq:rate}
\end{eqnarray}

\begin{remark}
The original interaction model presented in \cite{Bouchaud_etal_JStatMech14} considers the symmetric interaction rate with interactions rules $(\phi_i,\phi_j) \to (\phi_i',\phi_j')$ given by 
\begin{eqnarray*} 
 \phi'_i = \phi_i + \gamma (\phi_j - \phi_i) + \eta_i, \\ \phi'_j = \phi_j + \gamma (\phi_i - \phi_j) + \eta_j, 
\end{eqnarray*}
$0 \leq \gamma \leq \frac{1}{2}$. However, the kinetic equation for this system corresponds precisely to  \eqref{eq:rate} (up to a factor $2$ due to the rate being counted twice). This is precisely a consequence of the symmetry of the interactions. Therefore, for the following analysis one can consider either of the two interacting processes. 
\end{remark}

\subsection{Approximation under weak consensus assumption}
\label{sec:weak_consensus}
Next we assume that the \textit{consensus parameter} $\gamma$ is small, meaning that there is weak consensus during interactions. This is analogous to the so-called `grazing collisions' in gas dynamics \cite{Degond_Lucquin_1992}. To prevent the noise term to dominate the dynamics, we assume too that $\Sigma^2$ is also small and of the same order of magnitude. In particular, we define $\kappa$ as:
\begin{equation} \label{eq:def_kappa}
 \Sigma^2 = \kappa \gamma, \qquad \kappa = \mbox{Constant}. 
 \end{equation}

Under these assumptions, we Taylor expand the bracket in \eqref{eq:rate_inhomo} to obtain:
\begin{eqnarray}
&&\hspace{-1cm}
\frac{d}{d t} \int_{\R} f g \, d \phi = \gamma \int_{\R^2} \Big[   g' (\phi) \, (\psi-\phi) + \frac{\kappa}{2} \,  g'' (\phi)  \Big] \nonumber \\
&&\hspace{5cm}
 f(\phi,t) \, f(\psi,t) \,  \frac{G_\zeta\lp |\phi-\psi|\rp}{\Hop(G_\zeta*f)} \, d \phi \, d \psi, 
\label{eq:grazing}
\end{eqnarray}
where we have dropped the terms $\mathcal{O}(\gamma^2)$, $\mathcal{O}(\Sigma^2)$ and the ones with higher moments of $q$.
The equation for $f$ obtained after dropping the low order terms approximates the evolution of the opinion dynamics under the weak consensus assumption. This is in line with the grazing collision approximation in gas dynamics.

 Using integration by parts, we get the strong form of the Fokker-Planck equation for~$f$: 
\begin{eqnarray}
&&\hspace{-1cm}
\frac{\partial f}{\partial t}  =  Q(f), \label{eq:kineticFP} \\
&&\hspace{-1cm}
Q(f) = \gamma\, \partial_\phi \left\{ \frac{( \phi G_\zeta ) * f }{\Hop( G_\zeta*f)} \, f + \frac{\kappa}{2} \partial_\phi \lp \frac{(G_\zeta*f)}{\Hop(G_\zeta*f)} \, f \rp \right\}.
\label{eq:FP}
\end{eqnarray}

\begin{remark} 
\label{rem:SDE}
We can directly derive this equation from a stochastic interacting particle system. 
We introduce the following system for the intention $\phi_i(t)$ of agent $i$:
\begin{eqnarray*}
 d\phi_i &=& \gamma  \left[  \frac{\frac{1}{N}\sum_{j=1}^N (\phi_j - \phi_i) G_\zeta (\phi_j - \phi_i)}{\Hop\lp\frac{1}{N}\sum_{k=1}^N G_\zeta(|\phi_i-\phi_j|) \rp}\right] \, dt + \sqrt{\gamma\kappa \lp \frac{\frac{1}{N}  \sum_{j=1}^N G_\zeta (\phi_j - \phi_i)}{\Hop\lp\frac{1}{N}\sum_{k=1}^N G_\zeta(|\phi_i-\phi_j|) \rp}\rp} \ dB_t^i , 
 \end{eqnarray*}
 where $B_t^i$ are independent standard Brownian motions. We assume that the initial data $\phi_i(0)$ are drawn independently out of identical probability densities $f_0(\phi)$. We notice that the system is invariant under permutations of indices $i$. Hence the $\phi_i(t)$ are indistinguishable. Introducing the one and two-particle marginal distributions $f_N^{(1)}(\phi)$ and $f_N^{(2)}(\phi,\psi)$, and denoting by $g(\phi)$ any test function, we have:
\begin{eqnarray*}
&&\hspace{-1cm}
\frac{d}{dt} \int_{\mathbb R} g(\phi) \, f_N^{(1)}(\phi,t) \, d \phi = \\
&&\hspace{1cm}
= \gamma \frac{N-1}{N} \int_{{\mathbb R}^2} \Big[ g'(\phi) \, (\psi-\phi) \, + \frac{\kappa}{2} g''(\phi) \Big] \frac{G_\zeta (\psi-\phi)}{\Hop(G_\zeta *f)} \, f_N^{(2)}(\phi, \psi,t) \, d \phi \, d \psi. 
\end{eqnarray*}
Assuming propagation of chaos in the limit $N \to \infty$, namely there exists a one-particle distribution $f(\phi,t)$ such that $f_N^{(1)}(\phi,t) \to f(\phi,t)$ and $f_N^{(2)}(\phi, \psi,t) \to f(\phi,t)f(\psi,t)$, we get
\begin{eqnarray*}
&&\hspace{-1cm}
\frac{d}{dt} \int_{\mathbb R} g(\phi) \, f(\phi,t) \, d \phi = \\
&&\hspace{1cm}
= \gamma \int_{{\mathbb R}^2} \Big[ g'(\phi) \, (\psi-\phi) \, + \frac{\kappa}{2} g''(\phi) \Big] \frac{G_\zeta (\psi-\phi)}{\Hop(G_\zeta *f)} \,  f(\phi,t)f(\psi,t) \, d \phi \, d \psi, 
\end{eqnarray*}
which is the same as (\ref{eq:grazing}) without the low order terms.

\end{remark}

We can write (\ref{eq:FP}) in a more convenient way by introducing a potential $V_f(\phi)$ such that 
\begin{eqnarray}
&&\hspace{-1cm}
\partial_\phi V_f = \frac{( \phi G_\zeta ) * f}{G_\zeta*f}, 
\label{eq:potential}
\end{eqnarray}
and a Gibbs distribution
\begin{eqnarray}
&&\hspace{-1cm}
M_f (\phi)= \frac{1}{Z_f} \exp \Big( - \frac{2 V_f(\phi)}{\kappa} \Big), \qquad Z_f = \int_{\R} \exp \Big( - \frac{2 V_f(\phi)}{\kappa} \Big) \, d \phi. 
\label{eq:gibbs}
\end{eqnarray}
We recast (\ref{eq:FP}) into:
\begin{eqnarray}
&&\hspace{-1cm}
Q(f) = \gamma\frac{\kappa}{2}\, \partial_\phi \left\{ M_f \, \partial_\phi \lp \frac{(G_\zeta*f)}{\Hop(G_\zeta*f)} \, \frac{f}{M_f} \rp \right\}, 
\label{eq:FP2}
\end{eqnarray}
which in the symmetric case ($\Hop\equiv 1$) corresponds to
$$Q(f) = \gamma\frac{\kappa}{2} \,\partial_\phi \left\{ M_f  \partial_\phi \lp \frac{(G_\zeta*f)\, f}{M_f} \rp \right\}, $$
and for the non-symmetric case ($\Hop=\mbox{Id}$):
\begin{equation}\label{eq:Q_asym}
Q(f) = \gamma\frac{\kappa}{2}\, \partial_\phi \left\{ M_f  \partial_\phi \lp \frac{f}{M_f} \rp \right\}. 
\end{equation}
Remember that $\kappa \gamma = \Sigma^2$ which is assumed to be small.

\subsection{Equilibria and phase transition}
\label{sub:equi_sym}

In this section we investigate the equilibria for the operator $Q$ in Eq. \eqref{eq:FP2}.
In the general case, we have the

\begin{proposition}
$f_{\mbox{\scriptsize eq}}$ is an equilibrium, i.e. a solution of $Q(f_{\mbox{\scriptsize eq}}) = 0$, if and only if $f$ satisfies
\begin{itemize}
\item[i)] in the symmetric case:
\begin{eqnarray}
&&\hspace{-1cm}
(G_\zeta*f) \, f = B_f M_f, \quad  B_f =  \int_{\R} (G_\zeta*f) \, f \, d \phi;
\label{eq:equi}
\end{eqnarray}
\item[ii)] in the non-symmetric case:
\begin{equation}
f = M_f.
\label{eq:equi_asym}
\end{equation}
\end{itemize}
\label{prop:equi_gene}
\end{proposition}

\medskip
\noindent
\begin{proof} Integration by parts leads to 
\begin{eqnarray*}
&&\hspace{-1cm}
\int_{\R} Q(f) \frac{(G_\zeta * f)}{\Hop(G_\zeta *f)} \, \frac{f}{M_f} \, d \phi = -\gamma \frac{\kappa}{2} \int_{\R} M_f \, \Big| \partial_\phi \Big( \frac{(G_\zeta*f)}{\Hop(G_\zeta *f)} \, \frac{f}{M_f} \Big) \Big|^2 \, d \phi.  
\end{eqnarray*}
Then $Q(f) = 0$ iff $\frac{(G_\zeta*f)}{\Hop(G_\zeta*f)} \, \frac{f}{M_f}$ is a constant. This constant is given by the constraint $\int M_f =1$.
\end{proof}

Proving the existence of a fixed point for Eqs. \eqref{eq:equi} and \eqref{eq:equi_asym} is in general challenging. We will focus our attention on the Gaussian  case where. 
\begin{equation}
G(u) = e^{-u^2/2}.
\label{eq:gGauss}
\end{equation}
 In this case, we have the
\begin{proposition}
Assume that the interaction rate $G$ is given by \eqref{eq:gGauss}. Then Gaussian  functions $f_{\mbox{\scriptsize eq}} = F_{\sigma, \varphi}$, with 
\begin{equation}
F_{\sigma, \varphi} (\phi) = \frac{1}{\sqrt{2 \pi} \sigma} \exp \big( - \frac{|\phi-\varphi|^2}{2 \sigma^2} \big) , 
\label{eq:FGauss}
\end{equation}
are equilibria, provided that,
\begin{itemize}
\item[i)] in the symmetric case, $\sigma=\sigma_s$ with
\begin{equation}
\sigma_s^2 =  \frac{\zeta^2}{2\lp\frac{\zeta^2}{\kappa} - 1\rp};
\label{eq:sigma}
\end{equation}
\item[ii)] in the non-symmetric case, $\sigma=\sigma_a$ with 
\begin{equation}
\sigma_a^2 =  \frac{\zeta^2}{\frac{2\zeta^2}{\kappa} -1}. 
\label{eq:sigma_asym}
\end{equation}
\end{itemize}
Therefore, Gaussian  equilibria exist only for $\kappa$ such that, 
\begin{itemize}
\item[i)] in the symmetric case,
\begin{equation}
\kappa < \kappa_{c,s}(\zeta)= \zeta^2;
\label{eq:kappa_c}
\end{equation}
\item[ii)] in the non-symmetric case,
\begin{equation} \label{eq:kappa_c_asym}
\kappa < \kappa_{c,a}(\zeta)= 2 \zeta^2.
\end{equation} 
\end{itemize}
\label{prop:equi_gauss}
\end{proposition}

\begin{remark}[Constraints and phase transition]
\label{rem:phase_transition}
\mbox{}
\begin{itemize}
\item The constraints in Eqs. \eqref{eq:kappa_c} and \eqref{eq:kappa_c_asym} give a critical value of $\kappa$ for which a phase transition takes place between existence and non-existence of Gaussian  equilibria. Notice that when $\kappa$ approaches its critical value the variance of the equilibria diverges. Therefore, opinions spread towards dissension. Otherwise, a consensus arises. On the contrary, if the noise strength $\Sigma^2 \to 0$ then $\sigma \to 0$ and we expect to converge to a Dirac-delta distribution. The phase transition that we observe  here corresponds to the one proven in  Ref. \cite{Bouchaud_etal_JStatMech14} when considering the grazing collision case. 

\item It is easy to check that $\sigma_s>\sigma_a$, therefore, opinions in the symmetric case are more spread around their mean than in the non-symmetric one. 

\item  Replacing the value of $\kappa$ in Eq.  \eqref{eq:def_kappa} we can rewrite Eqs. \eqref{eq:kappa_c} and \eqref{eq:kappa_c_asym} as
$$\Sigma^2 < c_0\,\zeta^2 \gamma$$
for $c_0=1$ in the symmetric case and $c_0=2$ in the non-symmetric case. Notice that for higher values of $\gamma$ and $\zeta$, we expect more consensus and, therefore, condensation of the opinions; the larger $\gamma$ is, the closer the agents get in their opinions and the larger the typical interaction range $\zeta$ is, the more interactions take place and the faster a consensus can be reached. The product of these two quantities must bound the strength of the noise $\Sigma^2$.

\medskip

Notice that in \cite{Bouchaud_etal_JStatMech14} the constraint for the symmetric case takes the form:
$$\Sigma^2<\zeta^2\gamma(1-\gamma).$$
This is consistent with our result as in Section \ref{sec:weak_consensus}
we assumed weak consensus interactions and drop terms of order $\gamma^2$.

\end{itemize}

\end{remark}

In the sequel we will use repeatedly the following identities (expressed in the notation of Eq. \eqref{eq:FGauss}):
\begin{eqnarray} \label{eq:product_Gaussian s}
F_{\sigma,0}F_{\eta,0} &=& \frac{1}{\sqrt{2\pi (\sigma^2+\eta^2)}}\, F_{\sqrt{\frac{\sigma^2\eta^2}{\sigma^2+\eta^2}},0}\, ,\\
F_{\sigma,\varphi}*F_{\eta,\mu} &=& F_{\sqrt{\sigma^2+\eta^2}, \varphi+\mu}\, , \label{eq:convolution_Gaussian s}\\
\partial_\varphi(f*g)&=&(\partial_\varphi f *g)=(f*\partial_\varphi g). \label{eq:convolution_derivative}
\end{eqnarray}

Notice that in the Gaussian  case \eqref{eq:gGauss} we can express $G_\zeta$ as:
\begin{equation} \label{eq:Gzeta}
G_\zeta= \sqrt{2\pi}\zeta\, F_{\zeta, 0}.
\end{equation}

\begin{proof}[Proof of Proposition \ref{prop:equi_gauss}] In the case \eqref{eq:gGauss} we have $\phi G_\zeta = - \zeta^2 \partial_\phi G_\zeta$ and so
$$ \frac{( \phi G_\zeta ) * f}{G_\zeta*f} = - \zeta^2 \partial_\phi \log (G_\zeta*f)$$
which implies that 
$$ V_f = - \zeta^2 \, \log (G_\zeta*f) + \mbox{Constant}. $$
This implies that 
\begin{equation} \label{eq:aux}
 M_f = \frac{1}{C_f} (G_\zeta*f)^{\frac{2 \zeta^2}{\kappa}}, \quad C_f = \int_{\R} (G_\zeta*f)^{\frac{2 \zeta^2}{\kappa}} \, d \phi . 
 \end{equation}
For the symmetric case, inserting this into \eqref{eq:equi}, we get:
$$ (G_\zeta*f) \, f = \frac{B_f}{C_f} (G_\zeta*f)^{\frac{2 \zeta^2}{\kappa}}, $$
or  
$$ f = \frac{B_f}{C_f} (G_\zeta*f)^{\frac{2 \zeta^2}{\kappa}  -1}. $$
Now, one can check using \eqref{eq:convolution_Gaussian s} and \eqref{eq:Gzeta} that $f =  F_{\sigma, \varphi}$ is a solution
provided that 
$$ \sigma^2 = \frac{\zeta^2 + \sigma^2}{\frac{2 \zeta^2}{\kappa}  -1}, $$
which leads to \eqref{eq:sigma}. The condition for $\sigma^2>0$ is that $\frac{\zeta^2}{\kappa} - 1 >0$ which leads to the constraint $\kappa < \zeta^2$.

The non-symmetric case is dealt with analogously by inserting expression \eqref{eq:aux} into \eqref{eq:equi_asym}.
\end{proof}

\section{Space inhomogeneous model and continuum limit}
\label{sec:inhomogenous}
In this section we consider the space inhomogeneous version of the opinion formation model and investigate the evolution of the mean opinion dynamic when interactions become localized in space. This is done in the spirit of hydrodynamic limits for kinetic equations.

\subsection{Derivation of the kinetic model}
\label{sec:symmetric_kinetic}

In this section, we assume that the agents are endowed with a spatial variable $\alpha \in {\mathbb R}^n$ (with $n \in \{1,2,3\}$)  which can be the geographical distance or any other social metric. Then, each agent labelled $i$ ($i \in \{1,\ldots, N\}$) is described by its location variable $\alpha_i(t)$ and its opinion $\phi_i(t)$ at time $t$. Inspired by Ref.  \cite{Bouchaud_etal_JStatMech14}, we can interpret $\alpha_i$  as the city of voter $i$. It interacts with voter $j$ with opinion $\phi_j$ and resident in city $\alpha_j$ according to rule \eqref{eq:agent} with rate 
$$\frac{G_\zeta(\phi_i - \phi_j) \, F_\varepsilon(\alpha_i - \alpha_j)}{\Hop\lp\frac{1}{N}\sum_{k=1}^N G_\zeta(\phi_i-\phi_k)F_\eps(\alpha_i-\alpha_k) \rp} ,$$ 
with 
$$ G_\zeta (\phi) = G\big( \frac{|\phi|}{\zeta} \big), \qquad F_\varepsilon(\alpha) = \frac{1}{\varepsilon^n}  F\big( \frac{|\alpha|}{\varepsilon} \big), $$
and $\Hop$ given in Eq. \eqref{eq:Hop}.
The parameter $\varepsilon$ is supposed to be a measure of the interaction range in the position variable $\alpha$. We subsequently normalize $F$ such that 
$$ \int_{\R^n} F(|\alpha|) \, d\alpha = 1, $$
and denote by $D$ the constant 
\begin{equation} \label{eq:D}
 D = \frac{1}{2n} \int_{\R^n} F(|\alpha|) \, |\alpha|^2 \, d\alpha. 
 \end{equation}

Let $f=f(\alpha, \phi, t)$ be the probability density of the agents in the $(\alpha,\phi)$ space at time $t$. Following the same reasoning as previously, we can write for any test function $g=g(\alpha, \phi)$: 
\begin{eqnarray}
&&\hspace{-1cm}
\frac{d}{d t} \int_{\R^{n+1}} f g \, d\alpha \, d \phi = \int_{\R^{2n+3}} \Big[ g \big(\alpha,\phi + \gamma (\psi-\phi) + \eta \big) - g (\alpha,\phi) \Big] \label{eq:rate_inhomo} \\
&&\hspace{2cm}
 f(\alpha, \phi,t) \, f(\beta, \psi,t) \, \frac{G_\zeta (|\phi-\psi|)}{\Hop\left[F_\eps G_\zeta *_{\alpha,\phi} f\right](\phi)} \, F_\varepsilon (|\alpha-\beta|) \, q(\eta) \, d \eta \, d \phi \, d \psi \, d\alpha \, d\beta. 
\nonumber
\end{eqnarray}

\bigskip
As done previously in Sec. \ref{sec:weak_consensus} for the spatially homogeneous case, we consider the case of weak consensus interactions where $\gamma$ and $\Sigma^2$ are small and of the same order of magnitude. Taylor expanding the bracket in the previous expression and dropping the low order terms we get
\begin{eqnarray}
&&\hspace{-1cm}
\frac{d}{d t} \int_{\R} f g \, d\alpha \, d \phi = \gamma \int_{\R^{2n+2}} \Big[   \partial_\phi g (\alpha, \phi) \, (\psi-\phi) + \frac{\kappa}{2} \,  \partial^2_\phi g (\alpha, \phi)  \Big] \label{eq:grazing_inhomo}\\
&&\hspace{2cm} 
 f(\alpha, \phi,t) \, f(\beta, \psi,t) \, \frac{G_\zeta (|\phi-\psi|)}{\Hop\left[F_\eps G_\zeta *_{\alpha,\phi} f\right](\phi)} \, F_\varepsilon (|\alpha-\beta|) \, d \phi \, d \psi \, d\alpha \, d\beta, 
\nonumber
\end{eqnarray}
where $\kappa$ is given by Eq. \eqref{eq:def_kappa}.

Using integration by parts, we get the strong form of the Fokker-Planck equation for $f$ in the inhomogeneous case: 
\begin{eqnarray}
&&\hspace{-1cm}
\frac{\partial f}{\partial t}  =  Q(f), \label{eq:kineticFP_inhomo} \\
&&\hspace{-1cm}
Q(f) = \gamma\,\partial_\phi \left\{ \frac{\big( F_\varepsilon \, \phi G_\zeta \big) *_{\alpha, \phi} f }{\Hop\left[F_\eps G_\zeta *_{\alpha,\phi} f\right]} \, f+ \frac{\kappa}{2} \partial_\phi \lp \frac{\big(F_\varepsilon \, G_\zeta\big)*_{\alpha, \phi}f}{\Hop\left[F_\eps G_\zeta *_{\alpha,\phi} f\right]}  \,  f \rp \right\},
\label{eq:FP_inhomo}
\end{eqnarray}
where $ *_{\alpha, \phi}$ denotes a convolution in both the $\alpha$ and $\phi$ variables. This system can also be directly derived from a stochastic differential system in exactly the same manner as in the homogeneous case, remark \ref{rem:SDE}. Details are omitted.

Now, we plan to investigate the $\varepsilon \to 0$ limit, i.e. we assume that the exchange of intentions is at leading order a local phenomenon in space. We note that for any arbitrary function $\varphi$, we have:
$$ F_\varepsilon  * \varphi = \varphi + \varepsilon^2 D \Delta_\alpha \varphi + {\mathcal O}(\varepsilon^4), $$
where $\Delta_\alpha$ denotes the Laplacian operator in the $\alpha$ variable. Using this expansion and introducting a diffusive time variable $t' = \varepsilon^2  t$, we get, dropping the primes: 
 \begin{equation}
\hspace{-1cm}
\frac{\partial f^\varepsilon}{\partial t} + \mbox{R}^\eps   = \frac{1}{\varepsilon^2} Q(f^\varepsilon) \label{eq:FP_scaled_0} 
\end{equation}
where the operator $Q$ is given in \eqref{eq:FP} and
\begin{itemize}
\item[i)] for the symmetric case
\begin{equation}
\mbox{R}^\eps=\mbox{R}^\eps_s=- \gamma D \,\partial_\phi \Big\{ \big( \phi G_\zeta * \Delta_\alpha f^\varepsilon \big) \, f^\varepsilon + \frac{\kappa}{2} \partial_\phi \Big( ( G_\zeta * \Delta_\alpha f^\varepsilon ) \, f^\varepsilon \Big) \Big\};
\end{equation}
\item[ii)] and for the non-symmetric case
\begin{equation}
\mbox{R}^\eps=\mbox{R}^\eps_a=-\gamma D\, \partial_\phi \left\{ \frac{(\phi G_\zeta) *f}{(G_\zeta*f)^2}\, f\, (G_\zeta*\Delta_\alpha f) - \frac{(\phi G_\zeta)* \Delta_\alpha f}{G_\zeta* f}\, f\right\}.
\end{equation}
\end{itemize}	
In both cases $*$ denotes a convolution with respect to the $\phi$ variable only. In the following section, we study the formal limit  as $\varepsilon \to 0$ for the Gaussian  case \eqref{eq:gGauss}.

\subsection{Derivation of the hydrodynamic model (Gaussian  case)}

In this section, we restrict ourselves to the Gaussian  case \eqref{eq:gGauss} and study the evolution of the mean opinion for Eq. \eqref{eq:FP_scaled_0} in the limit $\eps \to 0$. To compute this limit, the concept of collision invariant is key as it will be explained next. However, in the non-symmetric case the mean opinion is not a conserved quantity. This difficulty will be overcome by using a technique reminiscent of the classical Hilbert expansion \cite{Cergignani_illner}. Notice that the lack of conserved quantities is not new in the literature, see Ref. \cite{Degond_etal_PRSA14} where the authors use a novel technique to compute the hydrodynamic limit, different from what we do here.

\medskip

\noindent Firstly, we deduce the following:
\begin{lemma} Assume that $G_\zeta$ is given by \eqref{eq:Gzeta} with $\kappa$ satisfying \eqref{eq:kappa_c} in the symmetric case and \eqref{eq:kappa_c_asym} in the non-symmetric case. Assume, further, that the Gaussian  equilibria given by Prop. \ref{prop:equi_gauss} are the only equilibria of $Q$.
Let $f^\varepsilon$ be a solution of (\ref{eq:FP_scaled_0}) and suppose that $f^\varepsilon \to f$ as nicely as needed. Then, there exist two functions $\rho=\rho(\alpha,t)$ and $\varphi=\varphi(\alpha,t)$ such that 
\begin{equation}
f(\alpha,\phi,t) = \rho(\alpha,t) \, F_{\sigma,\,\varphi(\alpha,t)}(\phi). 
\label{eq:equi_inhomo}
\end{equation}
The functions $\rho=\rho(\alpha,t)$ and $\varphi=\varphi(\alpha,t)$ are respectively the agents' density and intention field at location $\alpha$ and time $t$, i.e. we have 
\begin{equation}
\int_{\R} f(\alpha,\phi,t) \, d\phi= \rho(\alpha,t), \qquad \int_{\R} f(\alpha,\phi,t) \, \phi \, d\phi= \rho(\alpha,t) \varphi(\alpha,t). 
\label{eq:moments}
\end{equation}
\label{lem:equi_inhomo}
\end{lemma}

\begin{proof} Letting $\varepsilon \to 0$ in (\ref{eq:FP_scaled_0}), we get that $f$ formally satisfies $Q(f) = 0$. Since $Q$ only operates with respect to the $\phi$ variable we deduce that $f$ is proportional to $F_{\sigma,\varphi}$ by Lemma \ref{prop:equi_gauss} with $\varphi$ possibly depending on $(\alpha,t)$. Furthermore, since $F_{\sigma,\varphi}$ is a probability distribution with respect to $\phi$, the proportionality coefficient is the density $\rho(\alpha,t)$. 
\end{proof}

\begin{remark}[The general case]
Notice that the equilibria defined via Eqs. \eqref{eq:equi} and \eqref{eq:equi_asym} (if they exist) are not necessarily defined by some moments of $f$ like the total mass, mean value or variance. This is the case in classical kinetic theory where the goal in the hydrodynamic limit is to find the evolution of these moments, which correspond also to conserved quantities of the system. However, here, in the general case, it is not clear that this is the case. A partial attempt to solve this problem can be found in \cite{consensusexsitence}.
\end{remark}

\medskip
The function \eqref{eq:equi_inhomo} is a local thermodynamical equilibrium of the system. Now, we need to find the collision invariants of $Q$, namely the functions $\chi(\phi)$ such that 
$$ \int_{\R} Q(f) (\phi) \, \chi(\phi) \, d\phi = 0, \qquad \forall f. $$
We have the

\begin{lemma}[Collision invariants]
The function $\chi(\phi) = 1$ is a collision invariant of the symmetric and non-symmetric cases. The function  $\chi(\phi) = \phi$ is a collision invariant of the symmetric case.
\label{lem:IC}
\end{lemma}

\begin{proof} Since $Q$ is a derivative with respect to $\phi$, by integration by parts, it is straightforward that $\chi(\phi) = 1$ is a collision invariant for both cases. Now, for the symmetric case, considering 
$\chi(\phi) = \phi$, we have, using integration by parts: 
\begin{eqnarray*}
&&\hspace{-1cm}
\int_{\mathbb R} Q(f)(\phi) \, \phi \,  d \phi = - \gamma \int_{{\mathbb R}^2} (\phi-\psi) \,  G_\zeta (\phi-\psi) \,  f(\phi) \, f(\psi) \, d \phi \, d \psi = 0, 
\end{eqnarray*}
because the function $(\phi,\psi) \to G_\zeta (\psi-\phi) \,  f(\phi) \, f(\psi)$ is invariant by exchange of $\phi$ and $\psi$ while the function $(\phi,\psi) \to (\psi-\phi)$ is changed in its opposite.  
\end{proof}

That $\chi(\phi) = 1$ and $\chi(\phi) = \phi$ are collision invariants corresponds to the conservation of the number of agents and that of the mean intention during an encounter, respectively.

\begin{lemma}[Collision invariant for the linearised operator]
\label{lem:invariant_linearised} Assume that the solution to \eqref{eq:kineticFP_inhomo} decays sufficiently fast as $|\phi|\to \infty$.
Then, under the assumptions of Lemma \ref{lem:equi_inhomo}, in the non-symmetric case, the primitive $\chi(\phi)=\int^{\phi} (G_\zeta*F_{\sigma,\varphi(\alpha,t)})(\phi')d\phi'$ is a collision invariant of the linearized collision operator, denoted $Lin_Q$, around $F_{\sigma,\varphi(\alpha,t)}$.
\end{lemma}
\begin{proof}
The linearized form of $Q$ in Eq. \eqref{eq:Q_asym}  around $F_{\sigma,\varphi(\alpha,t)}$ is given by:
\begin{eqnarray} \label{eq:LinQ}
Lin_{Q}(f)&=&\gamma\,\partial_{\phi}\Bigg( \frac{(\phi G_\zeta*f)F_{\sigma, \varphi(\alpha,t)}}{G_\zeta*F_{\sigma, \varphi(\alpha,t)}}+\frac{(\phi G_\zeta*F_{\sigma,\varphi(\alpha,t)})f}{G_\zeta*F_{\sigma, \varphi(\alpha,t)}}\\
&&\hspace{1.3cm}- \frac{(\phi G_\zeta*F_{\sigma,\varphi(\alpha,t)})(G_\zeta*f)F_{\sigma, \varphi(\alpha,t)}}{(G_\zeta*F_{\sigma, \varphi(\alpha,t)})^{2}} +\frac{\kappa}{2} \partial_{\phi}f\Bigg)
\end{eqnarray}

By assumption, the boundary terms vanish. 
Integrating against $\chi$ and using that $\chi$ is the primitive of $G_\zeta*F_{\sigma,\varphi(\alpha,t)}$, we obtain:
\begin{eqnarray*}
&&\hspace{-1cm}\int_{\R} Lin_{Q}(f)(\phi)\left[\int^{\phi} (G_\zeta*F_{\sigma,\varphi(\alpha,t)})(\phi')d\phi' \right]d\phi  \\
&&=-\gamma\int_{\R} \Bigg( \frac{(\phi G_\zeta*f)F_{\sigma, \varphi(\alpha,t)}}{G_\zeta*F_{\sigma, \varphi(\alpha,t))}}+\frac{(\phi G_\zeta*F_{\sigma,\varphi(\alpha,t)})f}{G_\zeta*F_{\sigma, \varphi(\alpha,t))}}\\
&&\hspace{3cm}- \frac{(\phi G_\zeta*F_{\sigma,\varphi(\alpha,t)})(G_\zeta*f)F_{\sigma, \varphi(\alpha,t))}}{(G_\zeta*F_{\sigma,\varphi(\alpha,t)})^{2}} + \frac{\kappa}{2} \partial_{\phi}f\Bigg)(\phi)  (G_\zeta*F_{\sigma,\varphi(\alpha,t)})(\phi) d\phi\\
&&\textit{(Here we use Fubini's theorem to change the order of integration...)} \\
&&=-\gamma\int_{\R} f \Bigg[ -(\phi G_\zeta*F_{\sigma,\varphi(\alpha,t)})(\phi)+(\phi G_\zeta*F_{\sigma,\varphi(\alpha,t)})(\phi)\\
&&\hspace{3cm}-G_\zeta*\left(\frac{(\phi G_\zeta*F_{\sigma,\varphi(\alpha,t)}) F_{\sigma, \varphi(\alpha,t)}}{(G_\zeta*F_{\sigma,\varphi(\alpha,t)})}\right) - \frac{\kappa}{2} \partial_{\phi}(G_\zeta*F_{\sigma,\varphi(\alpha,t)})(\phi)  \Bigg]d\phi\\
&&=-\gamma\int_{\R} f \left[ -G_\zeta*\left(\frac{(\phi G_\zeta*F_{\sigma,\varphi(\alpha,t)}) F_{\sigma, \varphi(\alpha,t))}}{(G_\zeta*F_{\sigma,\varphi(\alpha,t)})}\right) - \frac{\kappa}{2} \partial_{\phi}(G_\zeta*F_{\sigma,\varphi(\alpha,t)})(\phi)  \right]d\phi\\
&&=0,
\end{eqnarray*}
where we have integrated by parts. Using the exponential decay of $F$ at infinity, the boundary terms in the integration by parts vanish. The last equality is obtained as follows, using Eqs. \eqref{eq:convolution_Gaussian s} and \eqref{eq:Gzeta}:
\begin{eqnarray*}
G_\zeta*\left(\frac{(\phi G_\zeta*F_{\sigma,\varphi(\alpha,t)}) F_{\sigma, \varphi(\alpha,t)}}{(G_\zeta*F_{\sigma,\varphi(\alpha,t)})}\right)(\phi)&=&-\zeta^{2}\Bigg{(}G_\zeta*\left(\frac{\left(\partial_{\phi}\left(G_\zeta*F_{\sigma,\varphi(\alpha,t)}\right)\right) F_{\sigma, \varphi(\alpha,t)}}{(G_\zeta*F_{\sigma,\varphi(\alpha,t)})}\right)\Bigg{)}(\phi)\\
&=&-\zeta^{2}\Big( G_\zeta*\Big[\left(\partial_{\phi}\ln\left( G_\zeta*F_{\sigma,\varphi(\alpha,t)}\right)\right) F_{\sigma,\varphi(\alpha,t)}\Big]\Big)(\phi)\\ 
&=&\zeta^{2}\frac{1}{\sigma^2+\zeta^2}\Big{(}G_\zeta*((\phi-\varphi(\alpha,t)) F_{\sigma,\varphi(\alpha,t)}\Big{)}(\phi)\\
&=&-\zeta^{2}\frac{\sigma^2}{\sigma^2+\zeta^2}\Big{(}G_\zeta* \partial_\phi F_{\sigma,\varphi(\alpha,t)}\Big{)}(\phi)\\
&=&-\frac{\kappa}{2} \partial_{\phi}\Big{(}G_\zeta* F_{\sigma,\varphi(\alpha,t)}\Big{)}(\phi).
\end{eqnarray*}
In the last equality we used the value of $\sigma$ given in \eqref{eq:sigma_asym}.
\end{proof}

\medskip
Finally, we can proof the
\begin{theorem} Under the assumptions of Lemma \ref{lem:equi_inhomo},
as $\varepsilon \to 0$, the solution $f^\varepsilon$ of \eqref{eq:FP_scaled_0} formally converges to a local thermodynamical equilibrium \eqref{eq:equi_inhomo}, i.e., $f^\varepsilon(\alpha,\varphi,t) = \rho(\alpha, t) F_{\sigma, \varphi(\alpha,t)}+\mathcal{O}(\eps^2)$ in a suitable topology, where $\rho$ is constant in time and the density of opinion $\rho\varphi$ evolves according to the following diffusion equation in regions where $\rho>0$:
\begin{itemize}
\item[i)] in the  symmetric case
\begin{equation}
\frac{\partial}{\partial t}(\rho\varphi) = C_s \, \nabla_\alpha \cdot \lp \rho^2 \nabla_\alpha \varphi\rp, \quad C_s:=\gamma D \frac{\zeta^3}{(2\sigma_s^2 + \zeta^2)^{3/2}} 
\label{eq:varphi_diffusion_sym}
\end{equation}
where $\sigma$ is given in Eq. \eqref{eq:sigma};
\item[ii)] and in the non-symmetric case
\begin{equation}
\frac{\partial}{\partial t} (\rho\varphi) = \frac{C_a }{\rho}\, \nabla_\alpha \cdot \lp \rho^2 \nabla_\alpha \varphi\rp, \quad C_a:=\gamma D \frac{\zeta^2}{\sigma_a^2+\zeta^2},
\label{eq:varphi_diffusion_asym}
\end{equation}
where $\sigma$ is given by Eq. \eqref{eq:sigma_asym}. 
\end{itemize}
\label{thm:diffusion_approx}
\end{theorem} 

For the significance of this theorem we refer the reader to Sec. \ref{sec:interpretation_results}.

%
%

\begin{proof}[Proof of Theorem \ref{thm:diffusion_approx}]  From Lemma \ref{lem:equi_inhomo}, we deduce that $f=f(\alpha,\phi,t)=\rho(\alpha,t) F_{\sigma,\varphi(\alpha,t)(\varphi)}$. Notice that by assumption $G_\zeta$ is of the form \eqref{eq:Gzeta}.
We split the proof between the symmetric case and the non-symmetric case.\\
\paragraph{Symmetric case.} 
 Successively multiplying Eq. \eqref{eq:FP_scaled_0} by $1$ and $\varphi$ and using Lemma \ref{lem:IC} we get:
\begin{eqnarray}
&&\hspace{-1cm}
\frac{\partial \rho^\varepsilon}{\partial t} (\alpha,t) =  0, \label{eq:rho_eps} \\
&&\hspace{-1cm}
\frac{\partial}{\partial t}  (\rho^\varepsilon \varphi^\varepsilon) (\alpha,t) + \gamma\, D \int_{{\mathbb R}^2} (\phi-\psi) \,  G_\zeta (\phi-\psi) \,  \Delta_\alpha f^\varepsilon(\alpha, \psi,t ) \, f^\varepsilon(\alpha, \phi,t) \, d \phi \, d \psi = 0, 
\label{eq:phi_eps}
\end{eqnarray}
where $\rho^\varepsilon$ and $\varphi^\varepsilon$ are defined by (\ref{eq:moments}) with $f$ replaced by $f^\varepsilon$. Now, letting $\varepsilon \to 0$ and using (\ref{eq:equi_inhomo}), we get:
\begin{eqnarray}
&&\hspace{-1cm}
\frac{\partial \rho}{\partial t} (\alpha,t) =  0, \label{eq:rho} \\
&&\hspace{-1cm}
\frac{\partial}{\partial t}  (\rho \varphi) (\alpha,t) +\gamma\, D \int_{{\mathbb R}^2} (\phi-\psi) \,  G_\zeta (\phi-\psi) \,  \Delta_\alpha \big(\rho(\alpha,t) \, F_{\sigma,\,\varphi(\alpha,t)}(\psi)\big) \nonumber \\
&&\hspace{7cm}
\rho(\alpha,t) \, F_{\sigma,\,\varphi(\alpha,t)}(\phi) \, d \phi \, d \psi = 0, 
\label{eq:phi}
\end{eqnarray}
We compute:
\begin{eqnarray*}
&&\hspace{-1cm}
\nabla_\alpha \big(\rho(\alpha,t) \, F_{\sigma,\,\varphi(\alpha,t)}(\psi)\big)  = F_{\sigma,\,\varphi(\alpha,t)}(\psi) \left[ \nabla_\alpha \rho + \Big( \frac{\psi - \varphi}{\sigma^2}\Big) \rho \nabla_\alpha \varphi \right](\alpha,t) , 
\end{eqnarray*}
and 
\begin{eqnarray}
&&\hspace{-1cm}
\Delta_\alpha (\rho(\alpha,t) \, F_{\sigma,\,\varphi(\alpha,t)}(\psi)) \label{eq:laplacian_f}\\
&&\hspace{-0.5cm}
=  F_{\sigma,\,\varphi(\alpha,t)}(\psi) \Bigg[ \Delta_\alpha \rho - \frac{1}{\sigma^2} \rho |\nabla_\alpha \varphi|^2 \nonumber \\ 
&&\hspace{3cm}
+ \Big( \frac{\psi - \varphi}{\sigma^2}\Big) \big(\rho \Delta_\alpha \varphi  +2\nabla_\alpha \rho \cdot \nabla_\alpha \varphi  \big) + \Big( \frac{\psi - \varphi}{\sigma^2}\Big)^2 \rho |\nabla_\alpha \varphi|^2 \Bigg](\alpha,t). \nonumber 
\end{eqnarray}
Then, we have: 
\begin{eqnarray*}
&&\hspace{-1cm}
- \frac{1}{\gamma D} \frac{\partial}{\partial t}  (\rho \varphi) (\alpha,t) = \\
&&\hspace{-0.5cm}
\rho\,\Big[ \Delta_\alpha \rho - \frac{1}{\sigma^2} \rho |\nabla_\alpha \varphi|^2 \Big](\alpha,t) 
\int_{{\mathbb R}^2} (\phi-\psi) \,  G_\zeta (\phi-\psi) \,  F_{\sigma,\,\varphi(\alpha,t)}(\psi) \, F_{\sigma,\,\varphi(\alpha,t)}(\phi) \, d \phi \, d \psi \\
&&\hspace{-0.5cm}
+\rho\, \left[\rho \Delta_\alpha \varphi + 2\nabla_\alpha \rho \cdot \nabla_\alpha \varphi  \right](\alpha,t) 
\int_{{\mathbb R}^2} \frac{\psi-\varphi}{\sigma^2} \, (\phi-\psi) \,  G_\zeta (\phi-\psi) \,  F_{\sigma,\,\varphi(\alpha,t)}(\psi) \, F_{\sigma,\,\varphi(\alpha,t)}(\phi) \, d \phi \, d \psi \\
&&\hspace{-0.5cm}
+ \rho^2\, |\nabla_\alpha \varphi|^2(\alpha,t)  
\int_{{\mathbb R}^2} \lp\frac{\psi-\varphi}{\sigma^2}\rp^2 \, (\phi-\psi) \,  G_\zeta (\phi-\psi) \,  F_{\sigma,\,\varphi(\alpha,t)}(\psi) \, F_{\sigma,\,\varphi(\alpha,t)}(\phi) \, d \phi \, d \psi \, .
\end{eqnarray*}
The first integral is identically zero by antisymmetry. The second and third integrals are denoted by $I_2$ and $I_3$. To compute them, we introduce the change of variables $\bar \phi = \phi - \varphi(\alpha,t)$, $\bar \psi = \psi -  \varphi(\alpha,t)$ and we get (omitting  the bars for simplicity):
\begin{eqnarray*}
I_2 &=&  \int_{{\mathbb R}^2} \frac{\psi}{\sigma^2} \, (\phi-\psi) \,  G_\zeta (\phi-\psi) \,  F_{\sigma,\,0}(\psi) \, F_{\sigma,\,0}(\phi) \, d \phi \, d\psi, \\
I_3 &=&  \int_{{\mathbb R}^2}  \lp \frac{\psi}{\sigma^2}\rp^2 ( \phi-\psi) \,  G_\zeta (\phi-\psi) \,  F_{\sigma,\,0}(\psi) \, F_{\sigma,\,0}(\phi) \, d \phi \, d \psi=0.
\end{eqnarray*}
The integral $I_3$ is identically 0 by antisymmetry.
Now, we compute the integral $I_2$ using consecutively Eqs. \eqref{eq:Gzeta}, \eqref{eq:convolution_derivative}, \eqref{eq:convolution_Gaussian s} and \eqref{eq:product_Gaussian s}:
\begin{eqnarray*}
I_2 &=&\int_{{\mathbb R}^2}    \zeta^2\partial_\phi G_\zeta (\phi-\psi) \,  \partial_\psi F_{\sigma,\,0}(\psi) \, F_{\sigma,\,0}(\phi) \, d \phi \, d\psi\\
&=& \zeta^2 \int_{\R} \lp \partial_\phi G_\zeta * \partial_\phi F_{\sigma, 0}\rp(\phi) \, F_{\sigma, 0}(\phi)\, d\phi\\
&=& \zeta^2 \int_{\R} \partial^2_\phi\lp G_\zeta * F_{\sigma, 0}\rp(\phi) \, F_{\sigma, 0}(\phi)\, d\phi\\
&=& -\zeta^2 \int_{\R}\partial_\phi \lp G_\zeta * F_{\sigma, 0} \rp(\phi) \, \partial_\phi F_{\sigma,0}(\phi)\, d\phi\\
&=&-\zeta^2 \int_{\R} \sqrt{2\pi}\zeta \, \partial_\phi \lp F_{\sqrt{\sigma^2+\zeta^2},0} \rp(\phi)\, \partial_\phi F_{\sigma,0}(\phi)\, d\phi\\
&=& -\zeta^3 \sqrt{2\pi}\int_{\R}\frac{\phi^2}{\sigma^2(\sigma^2+\zeta^2)\sqrt{2\pi (2\sigma^2+\zeta^2)}}\,\, F_{\!\!\sqrt{\frac{\sigma^2(\sigma^2+\zeta^2)}{2\sigma^2+\zeta^2}},0}(\phi)\, d\phi\\
&=&- \frac{\zeta^3}{\lp 2\sigma^2+\zeta^2 \rp^{3/2}} \,.
\end{eqnarray*}

Finally:
\begin{eqnarray*}
&&\hspace{-1cm}
\frac{1}{\gamma D} \frac{\partial}{\partial t}  (\rho \varphi) (\alpha,t) = 
\frac{\zeta^3}{(2 \sigma^2 + \zeta^2)^{3/2}}  \rho\Big[\rho \Delta_\alpha \varphi + 2\nabla_\alpha \rho \cdot \nabla_\alpha \varphi \Big](\alpha,t) ,
\end{eqnarray*} 
which gives the result for the symmetric case.

\medskip 
\paragraph{Non-symmetric case.} 
We conclude that $\rho=\rho(\alpha)$ is independent of time analogously as for the symmetric case. To compute the equation for $\varphi=\varphi(\alpha,t)$ we will use the collision invariant for the linearised operator given in Lemma \ref{lem:invariant_linearised}. We multiply equation \eqref{eq:FP_scaled_0} by 
$$\chi = \chi(\phi)= \int^\phi \lp G_\zeta * F_{\sigma, \varphi(\alpha,t)}\rp(\phi')\, d\phi'$$
and obtain
\begin{eqnarray}
&&\hspace{-1cm}\int_{\R} \chi(\phi) \left[\frac{\partial f_{\varepsilon}}{\partial t } - \gamma D\partial_{\phi} \left\{ f_{\varepsilon}\lp \frac{\phi G_\zeta*\Delta_{\alpha} f_{\varepsilon}}{G_\zeta*f_\eps} - \frac{(\phi G_\zeta*f_\varepsilon)(G_\zeta*\Delta_{\alpha} f_{\varepsilon})}{(G_\zeta*f_{\varepsilon})^{2}} \rp\right\}\right](\phi) d\phi   \label{eq:limit_asym}\\
&&\hspace{3cm}= \frac{1}{\varepsilon^{2}}\int_\R Q(f_{\varepsilon})(\phi) \chi (\phi) d\phi .\nonumber
\end{eqnarray}
Now writing that $f_{\varepsilon} = \rho(\alpha,t)F_{
\sigma,\varphi(\alpha,t)}+\varepsilon^{2}f_{1} + o(\varepsilon^{2})$ (by assumption), we have that $ Q(f_{\varepsilon}) =  \varepsilon^{2} Lin_{Q}(f_{1}) + o(\varepsilon^{2})$, where $Lin_{Q}$ is given in \eqref{eq:LinQ}. Consequently, by Lemma \ref{lem:invariant_linearised}, it holds that
$$ \frac{1}{\varepsilon^{2}}\int Q(f_{\varepsilon})(\phi) \chi (\phi) d\phi \underset{\varepsilon \to +0}{\longrightarrow}0.$$

We are left with computing the left hand side of Eq. \eqref{eq:limit_asym}. Letting $\varepsilon \rightarrow 0$, keeping the highest order in $\varepsilon$ and integrating by parts, we obtain
\begin{eqnarray*}
0&=& \int_\R \chi(\phi)\frac{\partial f}{\partial t } \, d\phi+ \gamma D\int_\R (G_\zeta*F_{\sigma,\varphi(\alpha,t)})\,f\,\left[ \frac{\phi G_\zeta*\Delta_{\alpha} f}{G_\zeta*f} - \frac{(\phi G_\zeta*f)(G_\zeta*\Delta_{\alpha} f)}{(G_\zeta*f)^{2}} \right](\phi) \,d\phi \\
&:=&  I_{1}+I_{2}+I_{3} 
\end{eqnarray*}
 with   $f=\rho (\alpha)F_{\sigma,\varphi(\alpha,t)}$. Next we compute each one of the terms. For $I_1$ we obtain, using again Eqs. \eqref{eq:product_Gaussian s}, \eqref{eq:convolution_Gaussian s}, \eqref{eq:convolution_derivative}, \eqref{eq:Gzeta}
\begin{align*}
I_{1}&=  \int_\R \frac{\partial f}{\partial t}(\phi) \chi(\phi)\, d\phi\\
&= \rho(\alpha)\int_\R \frac{\partial F_{\sigma,\varphi(\alpha,t)}}{\partial t}(\phi)\, \chi(\phi)d\phi \\
&= \frac{\partial \varphi}{\partial t}(\alpha,t)\rho(\alpha)\int_{\R} \partial_\phi F_{\sigma,\varphi(\alpha,t)}(\phi)\, \chi(\phi)d\phi \\
&=\frac{\partial \varphi}{\partial t}(\alpha,t)\rho(\alpha)\int_{\R} F_{\sigma,\varphi(\alpha,t)}(\phi) \,( G_\zeta* F_{\sigma,\varphi(\alpha,t)})(\phi) d\phi  \\ 
&=\sqrt{2\pi}\zeta \frac{\partial \varphi}{\partial t}(\alpha,t)\rho(\alpha)\int_{\R} F_{\sigma,\varphi(\alpha,t)} (\phi) \,F_{\sqrt{\sigma^{2}+\zeta^{2}}, \varphi(\alpha,t)}(\phi) d\phi\\
&=\sqrt{2\pi}\zeta \frac{\partial \varphi}{\partial t}(\alpha,t)\rho(\alpha) \frac{1}{\sqrt{2\pi (2\sigma^{2}+\zeta^{2})}}    \int_{\R} F_{\sqrt{\frac{\sigma^{2}(\sigma^{2}+\zeta^{2})}{2\sigma^{2}+\zeta^{2}}},\varphi(\alpha,t)}(\phi) d\phi\\
&= \sqrt{2\pi}\zeta\frac{\partial \varphi}{\partial t}(\alpha,t)\rho(\alpha) \frac{1}{\sqrt{2\pi (2\sigma^{2}+\zeta^{2})}}.
\end{align*}
The integral
$$I_{2} = \gamma D\int_{\R} F_{\sigma,\varphi(\alpha,t)}(\phi)\,(\phi G_\zeta*\Delta_{\alpha} f)(\phi) d\phi= - \gamma D \frac{\zeta^3}{\lp 2\sigma^2+\zeta^2\rp^{3/2}}\, \left[\rho \Delta_\alpha \varphi + 2\nabla_\alpha \rho \cdot \nabla_\alpha \varphi \right]$$
was already computed in the symmetric case (observe that the factor $\rho$ is not present here), so we are left with computing
\begin{align*}
I_{3} &=- \gamma D\int_{\R} F_{\sigma,\varphi(\alpha,t)}(\phi) \lp \frac{\phi G_\zeta*F_{\sigma, \varphi(\alpha,t)}\lp G_\zeta*\Delta_{\alpha} \lp \rho(\alpha)F_{\sigma, \varphi(\alpha,t)}\rp \rp}{G_\zeta*F_{\sigma, \varphi(\alpha,t)}}\rp(\phi) d\phi\,.
\end{align*}
Using the computation for $\Delta_\alpha (\rho\, F_{\sigma, \varphi(\alpha,t)})$ in Eq. \eqref{eq:laplacian_f}, we can compute analogously as for the symmetric case. One can check that the only term that does not cancel by an antisymmetry argument (as done previously for the integrals $I_1$ and $I_3$ in the symmetric case) is
\begin{align*}
I_{3} &=\left[\rho \Delta_\alpha \varphi + 2\nabla_\alpha \rho \cdot \nabla_\alpha \varphi \right] \gamma D\int_{\R} F_{\sigma,\varphi(\alpha,t)}(\phi) \frac{(\phi G_\zeta*f)(G_\zeta*\partial_\phi F_{\sigma,\varphi(\alpha,t)})}{(G_\zeta*f)}(\phi)d\phi\\
&=-\zeta^{2} \left[\rho \Delta_\alpha \varphi + 2\nabla_\alpha \rho \cdot \nabla_\alpha \varphi \right]\gamma D\int_{\R} F_{\sigma,\varphi(\alpha,t)}(\phi) \frac{\partial_\phi\lp G_\zeta*F_{\sigma,\varphi(\alpha,t)}\rp\, \partial_\phi \lp G_\zeta* F_{\sigma,\varphi(\alpha,t)}\rp}{(G_\zeta*F_{\sigma,\varphi(\alpha,t)})}(\phi)d\phi\\
&=\zeta^{2} \left[\rho \Delta_\alpha \varphi + 2\nabla_\alpha \rho \cdot \nabla_\alpha \varphi \right]\gamma D\int_{\R} F_{\sigma,\varphi(\alpha,t)}(\phi)\, \partial_\phi \lp G_\zeta*F_{\sigma,\varphi(\alpha,t)}(\phi)\rp\frac{(\phi-\varphi)}{\sigma^{2}+\zeta^{2}}(\phi)d\phi\\
&=-\sqrt{2\pi}\zeta^{3}\left[\rho \Delta_\alpha \varphi + 2\nabla_\alpha \rho \cdot \nabla_\alpha \varphi \right]\gamma D\int_{\R} F_{\sigma,\varphi(\alpha,t)}(\phi) F_{\sqrt{\sigma^{2}+\zeta^{2}},\varphi(\alpha,t)}(\phi) \frac{(\phi-\varphi)^{2}}{(\sigma^{2}+\zeta^{2})^{2}}d\phi\\
&=-\sqrt{2\pi}\zeta^{3}\left[\rho \Delta_\alpha \varphi + 2\nabla_\alpha \rho \cdot \nabla_\alpha \varphi \right]\gamma D\int_{\R} \frac{(\phi-\varphi)^{2}}{(\sigma^{2}+\zeta^{2})^{2}}    \frac{F_{\sqrt{\frac{\sigma^{2}(\sigma^{2}+\zeta^{2})}{2\sigma^{2}+\zeta^{2}}}, \varphi(\alpha,t)}}{\sqrt{2\pi (2\sigma^{2}+\zeta^{2})}} (\phi)d\phi  \\
&=-\sqrt{2\pi}\zeta^{3} \left[\rho \Delta_\alpha \varphi + 2\nabla_\alpha \rho \cdot \nabla_\alpha \varphi \right]\frac{ \gamma\,D\,\sigma^{2}}{(\sigma^{2}+\zeta^{2})\sqrt{2\pi (2\sigma^{2}+\zeta^{2})}(2\sigma^{2}+\zeta^{2})}. 
\end{align*}
We conclude
$$\rho\frac{\partial \varphi}{\partial t} =  \gamma D\frac{\zeta^{2}}{\sigma^{2}+\zeta^{2}} \left[\rho \Delta_\alpha \varphi + 2\nabla_\alpha \rho \cdot \nabla_\alpha \varphi \right].$$

\end{proof}

\section{Conclusions}

In this work we have studied the long time dynamics of an opinion model in the limit of spatially localized interactions under the assumption of low consensus. The considered model is borrowed from \cite{Bouchaud_etal_JStatMech14} where it is used to model voter's intentions. We have investigated the dynamics under two different interaction rates: one given in the original paper \cite{Bouchaud_etal_JStatMech14} (symmetric binary interactions) and the other inspired from Ref. \cite{Motsch_Tadmor_JSP11, Motsch_Tadmor_preprint} (non-symmetric binary interactions). In particular, we show that the density of opinion $\rho \varphi$ is a conserved quantity  for the symmetric case but not for the non-symmetric case; this is a direct translation of the conservative (or non-conservative) properties of their respective kinetic equations. Moreover, in Cor. \ref{cor:consensus} we give criterion on the spatial density $\rho$ to decide for which one of the two rates  consensus is reached faster.

The interaction rates considered come with a modeling choice: the symmetric rate in  \cite{Bouchaud_etal_JStatMech14} assumes, for instance, that any two pairs of individuals have the same influence on each other; this is not the case for the non-symmetric rate, which assumes that individuals in a large group are more influential than solitary ones.  Other modelling assumptions can be consider, for example, we could assume transport of individuals between different spatial regions. In this manner, one could assume that individuals move towards places where larger clusters are formed (migration from the country side to the cities) or a rule where individuals  move to places where others think alike. As we can see, a variety of possible modelling assumptions can modify these opinion models in various directions. The validity of the model considered will depend on the actual situation that we wish to describe. In any case, the derivation of the macroscopic equations for these models leads, firstly, to the long time dynamics, and secondly, to understanding the relevance of a particular modeling choice and contrast it with other alternative choices. Here, we have illustrated this idea by comparing the conservative properties and the influence on the speed towards consensus for two different interaction rates.

Deriving macroscopic equations for these models has many benefits: it gives access to the long time dynamics in both a qualitative way (by analysing the dependence of the diffusion constants on the model parameters) and in a quantitative way (by enabling numerical simulations for a large number of agents and over a long time without suffering from the curse of complexity). It also allows us to explore the relevance of particular modeling choices by contrasting them with alternate choices (illustrated here by the striking differences between two different interaction models in their conservation properties and speed towards consensus). For the non-symmetric case, we were led to develop a theory beyond the state-of-the-art of kinetic theory which is mostly restricted to the conservative case (see however Ref. \cite{Degond3} for another instance of treatment of non-conservative interactions). Being able to treat non-conservative interactions is key towards the development of kinetic theory for social interactions as such interactions can't be associated with a conserved quantity in general.

\bigskip
\section*{Data statement}
No new data was generated in the course of this research.

\end{document}